\documentclass{amsart}
\theoremstyle{plain}
\newcommand{\im}{\operatorname{im}}
\newtheorem{thm}{Theorem}[section]
\newtheorem{lm}[thm]{Lemma}
\newtheorem{prop}[thm]{Proposition}
\theoremstyle{definition}
\newtheorem{re}[thm]{Remark}
\newtheorem{ex}[thm]{Example}
\newcommand{\ZZ}{{\mathbb Z}}
\newcommand{\RR}{{\mathbb R}}
\newcommand{\CC}{{\mathbb C}}
\newcommand{\PP}{{\mathbb P}}

\newcommand{\la}{\langle}
\newcommand{\ra}{\rangle}
\newcommand{\Wedge}{\bigwedge\nolimits}
\newcommand{\liea}[1]{\mathfrak{#1}}
\newcommand{\lieg}[1]{\mathrm{#1}}

\def\sgn{{\rm sgn}\,}
\newcommand{\Hom}{\mathrm{Hom}}

\begin{document}

\title{Singular lines of trilinear forms}
\author[J.~Draisma]{Jan Draisma}
\address[Jan Draisma]{
Department of Mathematics and Computer Science\\
Technische Universiteit Eindhoven\\
P.O. Box 513, 5600 MB Eindhoven, The Netherlands\\
and Centrum voor Wiskunde en Informatica, Amsterdam,
The Netherlands}
\thanks{Draisma was partially supported by the MSRI programme on
tropical geometry and by DIAMANT, an NWO mathematics cluster.}

\email{j.draisma@tue.nl}

\author[R.~Shaw]{Ron Shaw}
\address[Ron Shaw]{
Centre for Mathematics\\
University of Hull\\
Cottingham Road, Hull HU6 7RX, United Kingdom}

\email{r.shaw@hull.ac.uk}

\begin{abstract}
We prove that an alternating $e$-form on a vector space over a
quasi-algebraically closed field always has a singular $(e-1)$-dimensional
subspace, provided that the dimension of the space is strictly greater
than $e$. Here an $(e-1)$-dimensional subspace is called singular if
pairing it with the $e$-form yields zero.  By the theorem of Chevalley
and Warning our result applies in particular to finite base fields.
Our proof is most interesting in the case where $e=3$ and the space has
odd dimension $n$; then it involves a beautiful equivariant map from
alternating trilinear forms to polynomials of degree $\frac{n-1}{2}-1$.
We also give a sharp upper bound on the dimension of subspaces all
of whose $2$-dimensional subspaces are singular for a non-degenerate
trilinear form. In certain binomial dimensions the trilinear forms
attaining this upper bound turn out to form a single orbit under the
general linear group, and we classify their singular lines.
\end{abstract}
\maketitle

\section{Introduction and the main theorem} \label{sec:Intro}

While alternating bilinear forms on an $n$-dimensional vector space
$V$ are very well understood in terms of their ranks and orbits---the
forms of rank at most $2k$ form a Zariski-closed set in which those of
rank exactly $2k$ form a single orbit for each $k=0,\ldots,\lfloor n/2
\rfloor$---{\em trilinear} and higher alternating multilinear forms on $V$
are much harder to grasp. For instance, being of rank at most $k$, that
is, being expressible as the sum of at most $k$ decomposable alternating
forms, is no longer necessarily a closed condition. Even the generic rank
of trilinear forms is not known exactly, although tight asymptotic results
have recently been obtained \cite{Abo09}. As for orbits, trilinear forms
have been classified on spaces of dimension up to seven over arbitrary
fields \cite{Cohen88,Schouten31}, as well as in dimensions $8$ over the
complex or real numbers \cite{Djokovic83,Gurevich35}. In dimension $8$
there are 23 orbits over the complex numbers, and the Hasse diagram
of their orbit closures is known explicitly \cite{Djokovic83b}. For
trilinear forms on $\CC^9$ the number of orbits is infinite,
but the invariant ring of the action of $\lieg{SL}_9$ on them is well
understood---in particular, it is free---and this contributes to the classification in \cite{Vinberg88}. Beyond that,
there seems little hope of a full classification.

This paper settles a question, put forward as conjecture A in
\cite{Shaw08}, about the geometry of trilinear forms in arbitrary
dimension. To state our main result we introduce some notation and
terminology. Write $\la.,.\ra:V \times V^* \to K$ for the natural pairing
of $V$ with its dual $V^*$ to the ground field $K$, and $\Wedge^d V$
for $d$-th exterior power of $V$. Recall that for $e \geq d$ there is a
natural bilinear map $\Wedge^d V \times \Wedge^e (V^*) \to \Wedge^{e-d}
(V^*)$ determined by
\[ (v_1 \wedge \cdots \wedge v_d,y_1 \wedge \cdots \wedge y_e) 
\mapsto \sum_{\pi:[d] \to [e]} \sgn(\pi) \left( \prod_{i=1}^d \la
v_i,y_{\pi(i)} \ra \right)  \Wedge_{j \not \in \im(\pi)} y_j. \]
Here the sum is taken over all injections $\pi:[d]:=\{1,\ldots,d\} \to
[e]$, of which the sign $\sgn \pi$ is defined to be the sign of the unique
extension of $\pi$ to a permutation $\pi':[e] \to [e]$ that is strictly
increasing on $\{d+1,\ldots,e\}$. Moreover, the last wedge is taking in
order of increasing index $j$. For $d=e=1$ this pairing reduces to $\la
.,. \ra$, and we will use the latter notation for general $d \leq e$,
as well. Whenever $d=e$ the pairing $\la .,. \ra$ is a non-degenerate
$K$-valued pairing, by which we identify $(\Wedge^e V)^*$ with $\Wedge^e
(V^*)$. Elements of either of these spaces, or of the space of alternating
multilinear forms $V^e \to K$, are called {\em alternating $e$-forms
on $V$}.

Let $\omega$ be an alternating $e$-form. An element $\lambda \in
\Wedge^d V$ with $d \leq e$ is called {\em singular for $\omega$}
or {\em $\omega$-singular} if $\la \lambda, \omega \ra=0 \in
\Wedge^{e-d}V^*$. Similarly, a $d$-dimensional subspace $U$ of $V$
is called singular for $\omega$ if $\la \Wedge^d U,\omega \ra=\{0\}$,
that is, if the one-dimensional subspace $\Wedge^d U$ of $\Wedge^d V$
is spanned by an $\omega$-singular element. More generally, suppose that
$d,e,f$ are natural numbers with $f \leq e$. Then a $d$-dimensional
subspace $U$ of $V$ is called {\em $f$-singular for $\omega$} if $\Wedge^f
U$ consists entirely of $\omega$-singular elements, or, equivalently,
if every $f$-dimensional subspace of $U$ is $\omega$-singular. 
For $d<f$ this is automatically true, and for $f=d$ this reduces to the 
statement that $U$ is $\omega$-singular.

For instance, a vector $v \in V$ is singular for an alternating
bilinear form $\omega$ if and only if $\omega(v,w)=0$ for all $w$,
that is, if and only if $v$ lies in the radical of $\omega$. Similarly,
a two-dimensional subspace $U$ of $V$ is singular for a trilinear form
$\omega$ if and only if $\omega(u,u',v)=0$ for all $u,u' \in U$. In
projective terminology, as in \cite{Shaw08}, such $U$ are called {\em
singular lines}. We will use both projective terminology (point, line)
and vector space terminology (one-dimensional subspace, two-dimensional
subspace). 

Notice that there is some asymmetry in these notions, which we could
have avoided by allowing that $e < d$ and by calling a the {\em pair}
$\lambda \in \Wedge^d V,\ \omega \in \Wedge^e V^*$ singular. However,
in this paper we will be primarily interested in questions of the
following flavour: fixing an alternating $e$-form $\omega$, what can we
say about the $d$-singular subspaces of $V$ for some $d \leq e$? This
justifies the present notions.

\begin{thm}[Main theorem]
Let $K$ be a quasi-algebraically closed field, that is, every non-constant
homogeneous multivariate polynomial of degree less than the number of
its variables has a non-zero $K$-valued root. Let $e$ be an integer
with $e \geq 3$, and let $V$ be a vector space over $K$ of dimension
at least $e+1$. Then every alternating $e$-form on $V$ has a singular
$(e-1)$-dimensional space.
\end{thm}

The conclusion of the theorem holds in particular for finite fields,
which are quasi-algebraically closed by the Theorem of Chevalley and
Warning \cite{Chevalley36,Warning36}. Note that the statement is false
if $V$ has dimension $e$: an $e$-form spanning the one-dimensional
space $\Wedge^e V^*$ does not have singular $(e-1)$-spaces. Also, the
following construction shows that the statement is, in general, false
for trilinear forms over non-quasi-algebraically closed fields.

\begin{ex}
Consider a real Euclidean space $E$ of dimension $7$ and inner product
denoted by $\cdot$. It is known, see \cite{BrownGray67}, that there exist
vector cross products $a\times b\in V$ which are bilinear and which
satisfy the axioms
\begin{align}
a\times b\cdot a &  =0,\qquad a\times b\cdot b=0,\label{Axiom 1}\\
a\times b\cdot a\times b &  =(a\cdot a)(b\cdot b)-(a\cdot b)^{2}.\label{Axiom 2}%
\end{align}
It follows that $\omega(a,b,c):=a \times b\cdot c$
defines an alternating trilinear form on $E$, and (\ref{Axiom 2}) implies
that $a\times b\neq0$ for all linearly independent $a,b$. Hence there
are no $2$-dimensional $\omega$-singular subspaces.

Such exceptional alternating trilinear forms are of
great interest and are well-known, see for example \cite{Baez02}, to be related
to the composition algebra $\mathbb{O}$ of the real octonions. With respect to
an orthonormal basis $\{x_{1},\ldots,x_{7}\}$ of $E^*$ one such $\omega 
\Wedge^3 E^*$ is given by 
\begin{equation}
\omega:=f_{124}+f_{235}+f_{346}+f_{457}+f_{561}+f_{672}+f_{713},\label{t_Fano}%
\end{equation}
where $f_{ijk}=x_{i}\wedge x_{j}\wedge x_{k}.$ It is known, see \cite[Theorem
1]{Bryant87}, that the stabiliser $\lieg{GL}(E)_\omega$ of $\omega$ in $\lieg{GL}(E)$ is a subgroup of $\operatorname{SO}%
(E)\cong\operatorname{SO}(7)$ which is isomorphic to the compact exceptional
real Lie group $G_{2}$, and that $\lieg{GL}(E)_\omega$ acts transitively on the
set of $2$-dimensional vector subspaces of $E.$ Now, from (\ref{t_Fano}), the
linear form $\omega(e_{1},e_{2},.)$ is nonzero. Consequently, by the
afore-mentioned transitivity, for any $2$-dimensional space $\RR a
\oplus \RR b \subseteq E$ the form $\omega(a,b,.)$ is nonzero,
thus recovering the fact that $\omega$ has no singular lines. 
\end{ex}

This paper is organised as follows. In Section \ref{sec:Grassmann}
we collect some results on divided powers of alternating forms of
even degree, which we use in Section \ref{sec:Proof} to prove our main
theorem. It turns out that the proof is most interesting for trilinear
forms in odd dimensions $n$, where we prove that the singular lines either
sweep out the entire projective $(n-1)$-space or else a hypersurface of
degree $\frac{n-1}{2}-1$. Finally, in Section \ref{sec:TTSS} we study
$2$-singular subspaces for a trilinear form on a vector space $V$.
In particular, we give a sharp upper bound in terms of $\dim V$ on the
dimension of such subspaces (assuming that $\omega$ is non-degenerate),
and study the trilinear forms in certain binomial dimensions attaining
this bound.

\section*{Acknowledgments}
We thank Arjeh Cohen for useful discussions on the topic of this paper.

\section{Divided powers in the Grassmann algebra} \label{sec:Grassmann}

For an $n$-dimensional vector space $V$ over a field $K$ let $\Wedge
V=\bigoplus_{d=0}^n \Wedge^d V$ denote the Grassmann algebra of $V$. This
is an associative $K$-algebra in which the multiplication, denoted
$\wedge$, takes $\Wedge^d V \times \Wedge^e V$ into $\Wedge^{d+e}V$.
Let $e_1,\ldots,e_n$ be a basis of $V$, and for a $d$-element subset
$I=\{i_1<\ldots<i_d\}$ of $[n]$ write $e_I:=e_{i_1} \wedge \ldots \wedge
e_{i_d}$. These elements form a basis of $\Wedge^d V$. Now assume that
$d$ is even, and let $\omega \in \Wedge^d V$. For every natural number
$k$ we define an element $\omega^{(k)}$ of $\Wedge^{kd} V$ as follows.
Write $\omega=\sum_{I \subseteq [n], |I|=d} \alpha_I e_I$ and set
\begin{equation} \label{eq:DividedPower}
\omega^{(k)}:=\sum_{I \subseteq [n],\ |I|=kd}
\left(
	\sum_{\{I_1,\ldots,I_k\},\ \dot{\bigcup}_j I_j=I,\ 
	|I_j|=d} (\prod_j \alpha_{I_j}) 
	e_{I_1} \wedge \ldots \wedge e_{I_d} \right). 
\end{equation}
The second sum is over all unordered partitions of $I$ into $k$
$d$-element subsets. It is important that these partitions are taken
unordered, so that a permutation of the $I_j$ does not
yield further terms
in the second sum. Note that the expression being summed is well-defined
as interchanging two consecutive factors $e_{I_j}$s does not change the
sign of the wedge-product---here we use that $d$ is even.

\begin{lm} \label{lm:Power}
For even $d$ the map $\Wedge^d V \to \Wedge^{kd} V,\ \omega \mapsto
\omega^{(k)}$ has the following properties:

\begin{enumerate}

\item $\omega \wedge \omega \wedge \ldots \wedge \omega$, where the
number of factors is $k$, equals $(k!) \omega^{(k)}$;
\label{it:1}

\item the map $\omega \mapsto \omega^{(k)}$ does not depend on the choice
of the basis $e_1,\ldots,e_n$;  \label{it:2}

\item for any $K$-linear map $A:V \rightarrow W$ of vector spaces we
have $((\Wedge^d A) \omega)^{(k)}=(\Wedge^{kd}
A)(\omega^{(k)})$; and 
\label{it:3}

\item if $d=2$ and $\dim V=2k$, then $\omega^{(k)}$ is zero
if and only if $\omega$ does not have full rank.
\label{it:4}
\end{enumerate}
\end{lm}

\begin{proof}
Property \eqref{it:1} is obvious: multiplying by $k!$ has the same effect
as summing, in \eqref{eq:DividedPower}, over all {\em ordered} partitions.

Property \eqref{it:2} is clear in characteristic zero by property
\eqref{it:1}. Now if we express $e_1,\ldots,e_n$ by an invertible
matrix $g$ in a second basis $e_1',\ldots,e_n'$, then the fact
that $\omega^{(d)}$ does not change when $K$ has characteristic $0$
translates into identities among certain polynomial expressions over $\ZZ$
in $\det(g)^{-1}$ and the $g_{ij}$. These identities hold over any field,
which proves the basis-independence over any $K$.

The basis independence implies property \eqref{it:3}: choose a basis
$e_1,\ldots,e_m,\ldots,e_n$ of $V$ such that $e_{m+1},\ldots,e_n$ span
$\ker(A)$, and extend $Ae_1,\ldots,Ae_m$ to a basis of $W$. In these bases
it is trivial to verify that $((\Wedge^d A) \omega)^{(k)}=
(\Wedge^{kd} A) \omega^{(k)}$.

Property \eqref{it:4} also follows from basis independence.  Indeed,
one can choose a basis $e_1,\ldots,e_{2m},\ldots,e_{2k}$ of $V$ with $m
\leq k$ such that $\omega=\sum_{i=1}^m e_{2i-1} \wedge e_{2i}$. If $m<d$
then all terms in \eqref{eq:DividedPower} are zero. If $m=d$ then the
expression equals $e_1 \wedge \ldots \wedge e_{2d} \neq 0$.
\end{proof}

\begin{re}
\begin{enumerate}
\item We call $\omega^{(k)}$ the $k$-th {\em divided power} of
$\omega$.
\item If $d=2$ and $n=2k$, then the $k$-th divided power of $\omega$
is known as its {\em Pfaffian}.
\item In our application below, this lemma will be applied
to $V^*$.
\end{enumerate}
\end{re}

\section{Proof of the main theorem} \label{sec:Proof}

We first prove our main theorem for trilinear forms. Here we 
distinguish two cases, according to the parity of $\dim V$.

\begin{prop} \label{prop:TrilinearEven}
Let $V$ be a vector space of {\em even} dimension over any field and let
$\omega \in \Wedge^3 V^*$. Then every one-dimensional subspace of $V$
is contained in an $\omega$-singular two-dimensional subspace of $V$.
\end{prop}

\begin{proof}
For any one-dimensional subspace $\la u \ra$ of $V$ the alternating
bilinear form $\la u,\omega \ra \in \Wedge^2 (V^*)$ has rank at most
$\dim V-1$, as $u$ is in its radical. But the rank of an alternating
bilinear form is even, so the rank of $\la u,\omega \ra$ is at most
$\dim V-2$. Hence there exists a $u'$, linearly independent of $u$,
such that $\la u \wedge u',\omega \ra=0$.
\end{proof}

\begin{thm} \label{thm:TrilinearOdd}
Let $V$ be a vector space of {\em odd} dimension $n \geq 5$ over a field $K$
and let $\omega \in \Wedge^3 V^*$. Then the union of all $\omega$-singular
lines is either all of $V$ or a hypersurface defined by a homogeneous
polynomial in $K[V]$ of degree $(n-1)/2 - 1$.
\end{thm}

In particular, if $K$ is quasi-algebraically closed, then this
hypersurface contains $K$-rational points, since $(n-1)/2-1$ is greater
than zero and less than $n$, the number of variables.

\begin{proof}
For any non-zero $u \in V$ consider the alternating bilinear form
$\omega_u:=\la u,\omega \ra \in \Wedge^2 V^*$. This is an element of
$\Wedge^2(u^0) \subseteq \Wedge^2(V^*)$, where $u^0$ is the annihilator
of $u$ in $V^*$. Setting $k:=(n-1)/2$, the $k$-th divided power
$\omega_u^{(k)}$ of $\omega_u$ lies in the one-dimensional subspace
$\Wedge^{n-1}(u^0)$ of the $n$-dimensional space $\Wedge^{n-1} (V^*)$.
By choosing a basis in the one-dimensional space $\Wedge^n(V^*)$ the
space $\Wedge^{n-1}(V^*)$ can be identified with $(V^*)^*=V$. Under
this identification the one-dimensional subspace $\Wedge^{n-1}(u^0)$
corresponds to the one-dimensional subspace $Ku$, and hence
$\omega_u^{(k)}$ corresponds to a multiple $f_\omega(u)u$ of $u$.
Now $f_\omega(u)$ is either zero or a homogeneous polynomial in $u$
of degree $k - 1=(n-1)/2-1$, which as $n>3$ is strictly positive. Its
non-zero roots are precisely the vectors $u \neq 0$ for which $\omega_u$
does not have full rank, by property \eqref{it:4} in Lemma \ref{lm:Power}
applied to the even-dimensional space $u^0$.  These, in turn, are
precisely the vectors $u \neq 0$ for which there exists a $u' \in
V$, linearly independent of $u$, for which $\la u \wedge u',\omega
\ra=0$---that is, the vectors $u \neq 0$ lying in some two-dimensional
$\omega$-singular space.
\end{proof}

Now we can prove the main theorem in full generality.

\begin{proof}[Proof of the main theorem.]
Let $\omega$ be an alternating $e$-form on a space of dimension larger
than $e$, and assume that $e \geq 3$. We have to prove that there
exist $(e-1)$-dimensional $\omega$-singular spaces. Choose an
$(e-3)$-dimensional subspace $U$ of $V$, let $\lambda \in \Wedge^{e-3}
V$ span $\Wedge^{e-3} U$, and consider $\omega':=\la \lambda,\omega
\ra \in \Wedge^3 (V/U)^*$. By Proposition~\ref{prop:TrilinearEven}
and Theorem~\ref{thm:TrilinearOdd} the space $V/U$, which is of dimension greater
than $3$, contains an $\omega'$-singular two-dimensional space $V'$. The
pre-image of $V'$ in $V$ is an $(e-1)$-dimensional $\omega$-singular
space.
\end{proof}

\begin{re}
The following remarks all concern trilinear forms.
\begin{enumerate}
\item The map $\Wedge^3 V^* \to S^{(n-1)/2-1} V^*$ sending $\omega$
to $f_\omega$ is $\lieg{GL}(V)$-equivariant by construction. This map
may prove useful in the further study of alternating trilinear forms.

\item If $K$ is finite and $n$ is odd, the theorem of Chevalley and
Warning allows one to add another $(n+1)/2$ linear equations, which then
still have a non-zero common root with $f$. Hence every space of vector
dimension $(n-1)/2$ intersects some singular line.

\item Suppose that $K$ is algebraically closed. Then every line
intersects some singular lines. If $f$ is non-zero, then a general line
has $(n-1)/2-1$ intersections with singular lines.

\item From the classification in \cite{Cohen88} one can deduce that for
trilinear forms on spaces of dimensions $5$ and $7$ the polynomial $f$ is
identically zero if and only if $\omega$ has a singular one-dimensional
space, that is, if and only if $\omega \in \wedge^3 U^*$ for some
proper subspace $U^*$ of $V^*$. The implication $\Rightarrow$ clearly
always holds, but the converse does not. Indeed, consider
the form 
\[ \omega=x_1 \wedge x_2 \wedge x_3
+ x_4 \wedge x_5 \wedge x_6
+ x_7 \wedge x_8 \wedge x_9, \]
where $x_1,\ldots,x_9$ are a basis of a $9$-dimensional space $V^*$. For
general $v$ the radical of $\omega_v$ is three-dimensional, hence
$f_\omega$ is identically zero, but $\omega$ does not have a singular
point.

\item In the previous example $\omega$ equals $\omega_1+\omega_2+\omega_3$
for a suitable decomposition $V=V_1 \oplus V_2 \oplus V_3$
and $\omega_i
\in V_i^*=(V_j \oplus V_k)^0$ for all distinct $i,j,k$. One may be led to
think that $f_\omega$ is identically zero if and only if $\omega$ is the
sum of forms $\omega_i$, where each $\omega_i \in V_i^*=(\bigoplus_{j \neq
i} V_j)^*$ for some non-trivial vector space decomposition $V=\bigoplus_i
V_i$. This is, however, not true: take $V$ equal to a simple Lie algebra
of odd dimension $n$ and rank $l$, say in characteristic zero. For
instance, one may take $V=\liea{sl}_m$ with $m$ even, so that $n=m^2-1$
and $l=m-1$ are odd. Let $\omega$ be the trilinear form on $V$ defined
by $\omega(u,v,w)=\kappa([u,v],w)$, where $[.,.]$ is the Lie bracket and
$\kappa$ is the Killing form. This form is alternating as the Killing form
is invariant ($\kappa([u,v],w)+\kappa(v,[u,w])=0$) and the Lie bracket is
alternating. Now for all $u$ the space of elements $v$
having zero Lie bracket with $u$
has dimension at least $l$. Hence if $l>1$, then the alternating bilinear
form $\omega_u$ has a radical. We conclude that $f_\omega=0$. On the other
hand, $\omega$ cannot be split as a sum of $\omega_i$s as above. Indeed,
$\omega$ does not have singular one-dimensional spaces, as $\kappa$
is a non-degenerate symmetric bilinear form and the centre of $V$ is
trivial. Hence $\omega$ is non-degenerate in the sense of \cite{Hora04},
and by the results of that paper the finest decomposition of $V$ and of
$\omega$ as above would be unique. Then, since $\omega$ is $V$-invariant,
the $V_i$ would have to be ideals in $V$, which would contradict the
fact that $V$ is simple. Concluding, at present we have no better
geometric description for $f_\omega \equiv 0$ than ``the
union of all singular lines is $\PP V$''.
\end{enumerate}
\end{re}

\section{Two-singular subspaces for alternating trilinear forms} 
\label{sec:TTSS}

Recall that a subspace $U$ of a vector space $V$ is called $2$-singular
for an alternating trilinear form $\omega$ if all $2$-dimensional
subspaces of $U$ are $\omega$-singular; in particular, we consider to be
$2$-singular all subspaces of dimension at most one, as well as all
$\omega$-singular $2$-dimensional subspaces. Here we present a result on
the possible dimensions of such a space $U$. The kernel of $V \mapsto
\Wedge^2 V^*,\ v \mapsto \la v,\omega \ra$ is called the {\em radical}
of $\omega$; and $\omega$ is called {\em non-degenerate} if its radical
is trivial.

\begin{thm} \label{thm:ns2}
Assume that $\dim V \geq 3$ and let $s\geq 2$ be the natural number
for which $\binom{s}{2} < n:=\dim V \leq \binom{s+1}{2}$. Then no
non-degenerate trilinear form on $V$ can have a $2$-singular space of
codimension strictly smaller than $s$; but there exist non-degenerate
trilinear forms on $V$ having $2$-singular spaces of codimension exactly
$s$. Moreover, if $n=\binom{s+1}{2}$, then the non-degenerate trilinear
forms having a $2$-singular space of codimension $s$ form a single
$\lieg{GL}(V)$-orbit.
\end{thm}

Note that if $V$ is three-dimensional this theorem reduces to the known
fact that there exist non-degenerate trilinear forms on $V$, and that
these form a single orbit. For the next interesting case $n=\binom{4}{2}=6$
see Example~\ref{ex:n6} below.

\begin{proof}
Suppose that $U$ is a $2$-singular subspace for the non-degenerate
trilinear form $\omega$ on $V$. Then we have a linear map $U
\to \Wedge^{2}(V/U)^*,\ u \mapsto \la u,\omega \ra$, whose kernel is
contained in the radical of $\omega$, hence zero by assumption. Hence we
find that $r:=\dim U \leq \binom{n-r}{2}=\dim \Wedge^2(V/U)^*$, or $s'
\geq n-\binom{s'}{2}$ where $s':=n-r$, or $\binom{s'+1}{2} \geq n$,
so that the codimension $s'$ of $U$ is at least $s$, as claimed.

Now let $U$ be a subspace of $V$ of codimension $s$. For the remainder
of this proof it is convenient to choose a vector space complement $W$
of $U$ in $V$. We may then identify $W^*$ with the annihilator of $U$
in $V^*$, and vice versa.  Since $\dim U\leq\dim \Wedge^2 W^*$, there
exist injective linear maps $L:U\to \Wedge^2 W^*$.  In fact we may chose
such an injection $L$ to have the property that the intersection of
the radicals of all images $L(u)$ is trivial. For if $s$
is even, then we may take $L$ such that some $L(u)$ is a non-degenerate
alternating $2$-form, while if $s$ is odd, then $s,\binom{s}{2} \geq 3$ by the
dimension restriction on $V$ and we can ensure that $\im(L)$ contains
two alternating forms of rank $s-1$ whose radicals are distinct.

We also view $L$ as an element of $U^* \otimes \Wedge^2 W^*$ and hence as
an element $\omega=\omega_L$ of $\Wedge^3 V^*$ by means of the (injective)
linear map $U^* \otimes \Wedge^2 W^* \to \Wedge^3 V^*$ determined by
$\xi \otimes \zeta \mapsto \xi \wedge \zeta$.  Then $\omega$ has $U$ as a
$2$-singular subspace, and we claim that $\omega$ is non-degenerate. For
this we have to prove that the linear map $H: V \to \Wedge^2 V^*,\ v
\mapsto \la v,\omega \ra$ is injective. This $H$ maps $U$ into $\Wedge^2
W^*$ and $W$ into $U^* \otimes W^*$, considered as a subspace of $\Wedge^2
V^*$ by the injective linear map determined by $\xi \otimes \zeta \mapsto \xi
\wedge \zeta$. Since the two subspaces $\Wedge^2 W^*$ and $U^* \otimes
W^*$ of $\Wedge^2 V^*$ intersect trivially, the injectivity of $H$ is
equivalent to the joint injectivity of $H|_U$ and of $H|_W$. Now $H|_U=L$
is injective by assumption, and $H(w)=0$ implies that $w$ lies in the
radical of $L(u)$ for all $u \in U$, a contradiction to the choice of
$L$. This proves that $\omega$ is non-degenerate.

Finally suppose that $n=\binom{s+1}{2},$ so that $\dim
U=\binom{s}{2}$. Then we need to show that all non-degenerate trilinear
forms $\omega'$ on $V$ having a $2$-singular subspace of codimension $s$
are in the $\lieg{GL}(V)$-orbit of the form $\omega$ constructed above.
First we move a $2$-singular codimension-$s$ subspace for $\omega'$
to $U$ by an element of $\lieg{GL}(V)$. Then $\omega'$ determines a
linear isomorphism $L':U \to \Wedge^2 W^*$, and we still have the group
of upper triangular linear maps
\[ g=\begin{bmatrix} A & B \\ 0 & C \end{bmatrix} \in \lieg{GL}(V) 
=\lieg{GL}(U \oplus W)
\]
with $A \in \lieg{GL}(U)$, $B \in \Hom(W,U)$, and $C \in \lieg{GL}(W)$ to
move $\omega'$ to $\omega$. First we take $B=0$ and $C=I$ and observe that
acting with $g$ on $\omega'$ corresponds to replacing $L'$ by $L' \circ
A^{-1}$. Hence by taking $A=L^{-1}L'$ we move $\omega'$ such that $L'$
becomes equal to $L$.

Now $\omega,\omega' \in (U^* \otimes \Wedge^2 W^*) \oplus \Wedge^3 W^*$
have the same component $L$ in the first summand, but $\omega'$ may have a
non-zero component $\mu'$ in the second summand while $\omega$ does not.
Take $A=C=I$ in the element $g$ and verify that $g$ then acts trivially on
$U$ and on $W^*$, while it sends an element $\xi$ of $U^*$ to $\xi-\xi
\circ B \in U^* \oplus W^*=V^*$. Hence $g$ fixes $\mu'
\in \Wedge^3 W^*$ and maps $L$ to $L-L\circ B$, with the slight
abuse of notation that the latter expression stands for the {\em image} of $L
\circ B$ under the projection $W^* \otimes \Wedge^2 W^* \to \Wedge^3
W^*$. By surjectivity of $L$ we may choose $B$ such that this image
coincides with $\mu'$, so that $g$ maps $\omega'$ to $\omega$.
This completes the proof that $\omega'$ lies in the orbit of $\omega$.
\end{proof}

We conclude by determining the singular lines of $\omega$ in the orbit
described above. We think of $U$ as equal to $\Wedge^2 W^*$, and then
the alternating trilinear form $\omega$ is determined by
\begin{align*} 
&\omega(\mu_1,\mu_2,.)=0 \text{ for $\mu_1,\mu_2 \in \Wedge^2 W^*$,}\\
&\omega(\mu,w_1,w_2)=\mu(w_1,w_2) \text{ for $\mu \in \Wedge^2 W^*,\
w_1,w_2 \in W$, and}\\
&\omega(w_1,w_2,w_3)=0 \text{ for $w_1,w_2,w_3 \in W$.}
\end{align*}
In addition to the
$2$-dimensional subspaces of $U=\Wedge^2 W^*$ also the $2$-dimensional
subspaces of the form $K \mu_1 \oplus K(\mu_2+w_2)$ with $\mu_1,\mu_2
\in \Wedge^2 W^*$ and $w_2$ in the radical of $\mu_1$ are singular. We
claim that these are the only singular lines. Indeed, consider a
$2$-dimensional subspace of the form $K(\mu_1+w_1) \oplus K(\mu_2+w_2)$
with $w_1,w_2$ linearly independent. Then choose any alternating
bilinear form $\mu_3$ on $W$ such that $\mu_3(w_1,w_2) \neq 0$.  Then we
have $\omega(\mu_1+w_1,\mu_2+w_2,\mu_3)=\mu_3(w_1,w_2)\neq 0$, so the
line is non-singular. This argument also implies that $U$ is the only
codimension-$s$ subspace that is $2$-singular: any other subspace $U'$
with this property cannot have a projection along $U$ onto $W$ that
is more than $1$-dimensional, and hence $U'$ must intersect $U$ in a
codimension-$1$ subspace. But if $\mu+w \in U'$ with $w \neq 0$, then
the elements of $U \cap U'$ must all have $w$ in their radicals. The
space of alternating bilinear forms on $W$ having $w$ in their radicals
is $\Wedge^2 (W/Kw)^*$ and has dimension $\binom{s-1}{2}$.  Hence this
space cannot contain a codimension-$1$ subspace of $\Wedge^2 W^*$.

\begin{ex} \label{ex:n6}
In the last part of Theorem \ref{thm:ns2} the smallest
dimension of interest is $n=6,$ a representative of the single
$\operatorname*{GL}(V)$-orbit being the form
\[
\omega=x_{2}\wedge x_{3}\wedge x_{4}+x_{1}\wedge x_{3}\wedge x_{5}+x_{1}\wedge x_{2}\wedge x_{6},
\]
for which the 3-dimensional subspace $U:=\prec e_{4},e_{5},e_{6}\succ$
is the unique 2-singular subspace of codimension $s=3.$ In this
example the map $L:U \to \Wedge^2 W^*$ in the preceding proof is
chosen to be that which sends $e_4,e_5,e_6$ to 
$x_2 \wedge x_3,x_1 \wedge x_3,x_1 \wedge x_2$, 
respectively.  As pointed out in \cite[Section 3]{Shaw08}, in the
case $K=\operatorname*{GF}(2)$ a trilinear form belonging to the
same orbit as $\omega$ arises from the cubic equation of the 35-set
$\psi\subset\operatorname*{PG}(5,2)$ supporting a non-maximal partial
spread $\Sigma_{5}$ of five planes in $\operatorname*{PG}(5,2).$
The unique projective plane $U$ singled out as being 2-singular for
$\omega$ is in fact one of the planes of $\Sigma_{5},$ and can also be
picked out geometrically by the property that each of the seven planes
$\notin\Sigma_{5}$ which lie in $\psi$ meets $U$ in a line and meets
each of the four other planes $\in\Sigma_{5}$ in a point.
\end{ex}

%\bibliographystyle{plain}
%\bibliography{diffeq}

\end{document}